\documentclass[a4paper,10pt]{article}

\usepackage{amsmath,latexsym,amssymb}
\usepackage{graphicx}
\usepackage{bm}
\usepackage{comment}
\usepackage{epsfig}
\newcommand{\supp}{\mathrm{supp}}% Support of the measure
\newcommand{\ext}{\mathrm{ext}}
\newcommand{\inte}{\mathrm{int}}
\renewcommand{\Re}{\operatorname{Re}} % Real part
\DeclareMathOperator{\Var}{Var}

\providecommand{\keywords}[1]{\textbf{Keywords} #1}

\let\olddet\det
\renewcommand{\det}{\olddet\nolimits} % Determinant with right limits

\newtheorem{definition}{Definition}[section]
\newtheorem{assumption}[definition]{Assumption}
\newtheorem{theorem}[definition]{Theorem}
\newtheorem{lemma}[definition]{Lemma}
\newtheorem{proposition}[definition]{Proposition}

\newtheorem{problem}[definition]{Problem}

\newenvironment{remark}[1][Remark]{\begin{trivlist}
\item[\hskip \labelsep {\bfseries #1}]}{\end{trivlist}}

\newenvironment{proof}%
{\rm \trivlist \item[\hskip \labelsep{\bf Proof. }]}%
{\hspace*{\fill}$\Box$\endtrivlist}
{\rm \trivlist \item[\hskip \labelsep{\bf Proof}]}%
{\hspace*{\fill}$\Box$\endtrivlist}

\numberwithin{equation}{section}

\def\Xint#1{\mathchoice
{\XXint\displaystyle\textstyle{#1}}%
{\XXint\textstyle\scriptstyle{#1}}%
{\XXint\scriptstyle\scriptscriptstyle{#1}}%
{\XXint\scriptscriptstyle\scriptscriptstyle{#1}}%
\!\int}
\def\XXint#1#2#3{{\setbox0=\hbox{$#1{#2#3}{\int}$}
\vcenter{\hbox{$#2#3$}}\kern-.5\wd0}}

\def\dashint{\Xint-}

\usepackage{authblk}
\title{Universality for conditional measures of the sine point process}
\author[1]{Arno B.J. Kuijlaars}
\author[2]{Erwin Mi\~{n}a-D\'{\i}az}
\affil[1]{Katholieke Universiteit Leuven, Department of Mathematics,
Celestijnenlaan~200B box 2400, 3001 Leuven, Belgium.
E-mail:~{\tt arno.kuijlaars@kuleuven.be}
}
\affil[2]{The University of Mississippi,
Department of Mathematics,
Hume Hall 305,
P.~O.~Box 1848,
University, MS 38677-1848, USA.
E-mail:~{\tt minadiaz@olemiss.edu}}
\date{\today}                     %% if you don't need date to appear
\date{Journal of Approximation Theory 243 (2019), 1--24}
\begin{document}
\maketitle

\begin{abstract}
The sine process is a rigid point process on the real line, which
means that for almost all configurations $X$, the number of points
in an interval $I = [-R,R]$ is determined by the points
of $X$ outside of $I$. In addition, the points in $I$ are an
orthogonal polynomial ensemble on $I$ with a weight function that
is determined by the points in $X \setminus I$. We
prove a  universality result  that in particular implies that the 
correlation kernel of the orthogonal polynomial ensemble
tends to the sine kernel as the length $|I|=2R$ tends to infinity, thereby
answering a question posed by A.I. Bufetov.
\end{abstract}


\section{Introduction and statement of results}

\subsection{Introduction}
The aim of this paper is to prove a universality result for determinantal
point processes on the real line.
The result is that certain correlation kernels, associated with orthogonal
polynomial ensembles, tend to the
universal sine kernel 
\begin{equation} \label{def:sinekernel}
	 \frac{\sin \pi(x-y)}{\pi(x-y)},  
	 \end{equation}
see Theorems \ref{mainthm} and \ref{mainthm:wR} for the precise statements.
Such limiting results of orthogonal polynomial ensembles are well-known in 
the theory of random matrices, see e.g.\ \cite{DKMVZ1, KMVV, Lubinsky2016}. 
The  kernel \eqref{def:sinekernel} is the typical limiting  kernel for bulk eigenvalue 
correlations of 
random Hermitian matrices, also beyond invariant ensembles \cite{Erdos2013}.
The sine kernel  is the correlation kernel
for a translation invariant determinantal point process $\mathbb P_{\sin}$
on $\mathbb R$, normalized in such a way that neighboring points
have distance $1$ on average. 

The problem we address is motivated by recent studies on the sine  process
and other determinantal processes by Bufetov \cite{Bufetov16}. The 
starting point of \cite{Bufetov16} is the quasi-invariance of point processes,
originating from the work of Olshanski \cite{Olshanski11},
combined with the number rigidity of the sine process in the sense 
of Ghosh and Peres  \cite{Ghosh, GhoshPeres}.
The number rigidity (or simply rigidity) of a 
locally finite point process $\mathbb P$ on the real line
means the following. For any compact interval $I$, and for 
$\mathbb P$-almost all point configurations $X$, it is true that
$X \setminus I$  determines the number of points in $X \cap I$ almost surely.
Thus, by observing the points outside of $I$, one can deduce the number of
points inside $I$, with probability one. Number rigidity is a very 
remarkable property, which makes the sine process very different from 
a Poisson point process. Indeed, in a Poisson point process, the points in $I$ 
and those in $I^c = \mathbb R \setminus I$ are independent
and so from knowing the points  in $I^c$ one gains no extra information
about the points in $I$. For more examples of rigid point processes,
see \cite{Bufetov16-1,BufetovDabrowskiQiu15,BufetovShirai17,
GhoshKrishnapur}.

Number rigidity can be expressed in terms of conditional measures.
Let $I$ again be a compact interval, and let $Y$ be a locally
finite configuration of points in $I^c$.  The conditional measure
$\mathbb P( \cdot \mid Y; I^c)$ of a point process $\mathbb P$
is a new point process  that is supported on configurations $X$
with $X \cap I^c = Y$. Informally, it is the point process obtained
from $\mathbb P$ by conditioning that the points outside of $I$ coincide
with $Y$. See \cite{Bufetov16} or \cite[Chapter 5, Section 8]{Parth67} 
for a  precise description.  For a rigid point process $\mathbb P$, and 
for $\mathbb P$-almost all $X$ the conditional measure 
$\mathbb P(\cdot \mid X \setminus I; I^c)$ is identified with a
point process on $I$ with exactly $N = \#(X \cap I)$ points.

For a large class of rigid determinantal point processes 
$\mathbb P$ on $\mathbb R$, including the sine process,
Bufetov \cite[Theorems 1.1 and 1.4]{Bufetov16} showed that, 
for $\mathbb P$ almost all $X$, the conditional 
measure $\mathbb P(\cdot \mid X \setminus I; I^c)$, when considered
as a point process on $I$, has a joint density on $I^N$ of the form
\begin{equation} \label{eq:OPdensity} 
	\frac{1}{Z_{I,X}} \prod_{j<k} (t_k-t_j)^2 \prod_{j=1}^N \rho_{I,X}(t_j),
	\qquad N = \#(X \cap I), 
	\end{equation} 
for certain functions $\rho_{I,X}$ that in addition to $I$ and $X$ also 
depend on the point process.
The arguments leading to \eqref{eq:OPdensity} are based on  
quasi-invariance properties \cite{Bufetov14, Olshanski11} of the point
processes.
For the sine process these functions are such that, 
see \cite[Corollary 1.5]{Bufetov16}
\[ \frac{\rho_{I,X}(t_1)}{\rho_{I,X}(t_2)} = 
	\lim_{R \to \infty}  \prod_{p \in X \setminus I, \, |p| \leq R} 
		\left(  \frac{t_1-p}{t_2-p} \right)^2,
		\qquad t_1, t_2 \in I. 
	\]
Hence, one might take 
\begin{equation} \label{eq:rhoIX} 
	\rho_{I,X}(t) =  
	\prod_{p \in X \setminus I} \left( 1 - \frac{t}{p} \right)^2,
		\qquad t \in I, \end{equation}
with a product that converges in principal value.

A joint density of the form \eqref{eq:OPdensity} is determinantal with
a correlation kernel that is built out of the orthogonal polynomials 
for the weight $\rho_{I,X}$ on $I$. Namely, if $(\varphi_j)_{j=0}^{\infty}$
is the sequence of orthonormal polynomials, i.e., $\deg \varphi_j = j$ and
\[ \int_I \varphi_j(t) \varphi_k(t) \rho_{I,X}(t) dt = \delta_{j,k}, \]
then the kernel is
\begin{equation} \label{eq:kernelKIX} 
	K_{I,X}(x,y) = \sqrt{\rho_{I,X}(x) \rho_{I,X}(y)}
	\sum_{j=0}^{N-1} \varphi_j(x) \varphi_j(y). 
	\end{equation}

If we now take $I$ centered at the origin, it will be reasonable to expect  
that the effect of conditioning on $X \setminus I$ becomes less 
important as the length of $I$ tends to infinity. In the limit $|I| \to \infty$
we expect to recover the kernel of the point process $\mathbb P$.
For the sine process this was stated as an open problem by
A.I. Bufetov \cite{BufetovMay16} to one of us. 
\begin{problem}[Bufetov] \label{prob:Bufetov}
For a locally finite configuration $X$ on $\mathbb R$, and for an
interval $I$, let $K_{I,X}$ be the orthogonal polynomial kernel 
\eqref{eq:kernelKIX} associated 
with the weight  function \eqref{eq:rhoIX}.
Is it then true that for $\mathbb P_{\sin}$-almost all point configurations $X$ we 
have that
\[ K_{I,X}(x,y) \to   \frac{\sin \pi(x-y)}{\pi(x-y)} \]
 as $I = [-R,R]$ and $|I| \to \infty$?
\end{problem}

We are going to answer Problem \ref{prob:Bufetov} in the affirmative 
in a more general,
deterministic setting, that we will explain next.

\subsection{Notation and statement of main results} 

Our results deal with a fixed, deterministic set of points
$ X = \{ p_n \mid n \in \mathbb Z \}$.
Our assumptions on $X$ are
\begin{assumption} \label{assumptions}
\begin{enumerate}
\item[(a)] $(p_n)_{n \in \mathbb Z}$ is a strictly increasing doubly infinite sequence, indexed so that
\begin{equation} \label{eq:pnnumbering}
\cdots < p_{-2} < p_{-1} < 0 \leq p_0 < p_1 < \cdots
\end{equation}
\item[(b)] the series $\sum p_n^{-1}$ converges in principal value, that is, 
\begin{align}\label{convergence:pv}
\lim_{S\to\infty}\sum_{0< |p_n|<S}\frac{1}{p_n} \quad \text{ exists},
\end{align}
\item[(c)] and
\begin{align}\label{limit:pn/n}
\lim_{n\to\pm\infty}\frac{p_n}{n}=1.
\end{align}
\end{enumerate}
\end{assumption}
The convergence in principal value \eqref{convergence:pv} 
is then (in our situation, assuming \eqref{eq:pnnumbering} and \eqref{limit:pn/n}) 
the same as saying that
\begin{equation} \label{convergence:pv2} 
	\sum_{n=1}^{\infty} \left( \frac{1}{p_n} + \frac{1}{p_{-n}} \right) 
	\qquad \text{ converges}. 
	\end{equation}

The assumptions are satisfied for a typical configuration $X$ from
the sine process, since the sine process is a simple point process, and
if $X = \{ p_n \mid n \in \mathbb Z\}$ with $(p_n)$ as in \eqref{eq:pnnumbering} 
then it is known that (see Appendix \ref{appendixB}) 
\begin{equation} \label{eq:pnsineprocess} 
	p_n = n + O(|n|^{1/2} \log^2 |n|) \qquad \text{ as } n \to \pm \infty 
	\end{equation}
holds $\mathbb P_{\sin}$-almost surely. 
It is easy to show that \eqref{eq:pnnumbering} and \eqref{eq:pnsineprocess} 
imply  \eqref{convergence:pv} and \eqref{limit:pn/n}.

For a non-negative  weight function $w$ on $\mathbb{R}$ 
such that $\int |t|^{n} w(t)dt<\infty$ for all $n\ge 0$, we denote 
by $\left(\varphi_j( \cdot; w)\right)_{j=0}^\infty$ the sequence of orthonormal 
polynomials with respect to $w$. That is, $\varphi_j( \cdot; w)$ is a 
polynomial of degree $j$ with positive leading coefficient, and 
\[
\int \varphi_j( t; w)\varphi_k(t; w) w(t)dt=\delta_{jk}.
\]

The  kernel functions associated to $w$ are  defined by 
\begin{equation} \label{kernelKn} 
K_N(x,y;w) = \sqrt{w(x) w(y)} \sum_{j=0}^{N-1} \varphi_{j}(x;w) \varphi_j(y;w),
\quad N\geq 1. 
\end{equation}
In the proof of Theorem \ref{mainthm:wR} we will also consider the 
polynomial kernels
\begin{align}\label{kernelwithoutweights}
\widehat{K}_N(x,y;w) = \sum_{j=0}^{N-1} \varphi_{j}(x;w) \varphi_j(y;w),
\end{align}
that do not include the square root of the weight factors. 

With $X$ being fixed and $I = [-R,R]$, the weight \eqref{eq:rhoIX}
is equal to   
\begin{align} \label{eq:rhoR}
	\rho_R(t) & = \prod_{|p_n| > R} \left(1 - \frac{t}{p_n}\right)^2 \\	
	& = \lim_{S\to\infty}\prod_{R < |p_n|<S}\left(1-\frac{t}{p_n}\right)^2.
	\nonumber
\end{align}
The limit exists because of the assumption \eqref{convergence:pv}.
Because of \eqref{limit:pn/n} we have $N(R)/R \to 2$ as $R \to \infty$. 
The main result of this paper is the following universality theorem,
that also answers the Problem \ref{prob:Bufetov} of Bufetov.
\begin{theorem}\label{mainthm}
Let $(p_n)_{n \in \mathbb Z}$ be a doubly infinite sequence satisfying
Assumption \ref{assumptions}.  Let 
\begin{equation} \label{eq:defNR}
	N(R):= \# \{ p_n \mid |p_n| \leq R \}.
\end{equation}
Then we have  
\[ 
\lim_{R\to\infty}K_{N(R)} (x,y;\rho_R)= \frac{\sin \pi (x-y)}{\pi(x-y)} 
\]
uniformly for $x$ and $y$ in compact subsets of the real line.
\end{theorem}

To prove Theorem \ref{mainthm} we are going to rescale the weights to the 
interval $[-1,1]$.  We consider the new weights 
\begin{equation} \label{def:wR}
w_R(t):=\rho_R(Rt),
\end{equation}
so that $\sqrt{R}\varphi_j(t;\rho_R)=\varphi_j(t/R;w_R)$, and consequently,
\[
K_N(x,y;\rho_R)=\frac{1}{R} K_N\left(\frac{x}{R},\frac{y}{R};w_R\right).
\]
Since $N(R)/R\to 2$ as $R\to\infty$, 
Theorem \ref{mainthm} will follow as a corollary of the following result.

\begin{theorem}\label{mainthm:wR}
Let $(p_n)_{n \in \mathbb Z}$ be a doubly infinite sequence satisfying
Assumption \ref{assumptions}.   
Let $N$ be an integer depending on $R$ in such a way that 
\begin{equation} \label{eq:NRlimit}
	\lim_{R \to \infty} \frac{N}{R} = 2. 
\end{equation}
Then we have
\begin{align}\label{eq:KNuniv}
	\lim_{R\to\infty}\frac{2}{N}K_N\left(\frac{2x}{N},\frac{2y}{N};w_R\right)
	=\frac{\sin \pi (x-y)}{\pi(x-y)} 
\end{align}
uniformly for $x$ and $y$ in compact subsets of the real line.
\end{theorem}
	
\subsection{Outline of the proof}
A universality limit such as \eqref{eq:KNuniv} is well-known for varying
weights of the form $e^{-NV(t)}$, see \cite{DKMVZ1}, where $V$ is
real analytic on $\mathbb R$.
Suppose that the equilibrium measure $\mu_V$ corresponding to the external field 
$V$ has  a density $\psi_V$ (see \cite{Deift, SaffTotik} for main references
on equilibrium measures in external fields). 
If $\psi_V(0) > 0$, then  the universality  says that
\begin{equation} \label{eq:universalV} 
	\lim_{N \to \infty} \frac{1}{\psi_V(0)N} K_N \left(\frac{x}{\psi_V(0) N},
	 \frac{y}{\psi_V(0) N}; e^{-NV} \right)
	= \frac{\sin \pi(x-y)}{\pi(x-y)}. 
	\end{equation}

The external field $V$ that plays a role in this paper is
\begin{equation} \label{eq:Vdef} 
	V(t) = (1+t) \log(1+t) + (1-t) \log(1-t), \qquad t \in [-1,1]. 
	\end{equation}
It is easy to see by direct computation that 
\[ V(t) = \int_{-1}^1 \log|t-s| ds, \qquad \text{for } t \in [-1,1], \]
which means that the equilibrium measure for $V$ has the constant density 
\begin{equation} \label{eq:psiV} 
	\psi_{V}(t) \equiv 1/2, \qquad t \in [-1,1]. 
	\end{equation}
The universality limit \eqref{eq:universalV} holds for this $V$ and thus
takes the shape
\begin{equation} \label{eq:universalV0} 
	\lim_{N \to \infty} \frac{2}{N} K_N \left(\frac{2x}{N}, \frac{2y}{N}; 
	e^{-NV} \right)
	= \frac{\sin \pi(x-y)}{\pi(x-y)}. 
	\end{equation}
	
The limit \eqref{eq:KNuniv} for $w_R$ can be understood
from comparison with \eqref{eq:universalV0}, since it turns out
that $w_R \approx e^{-NV}$ in the sense that as $N,R \to \infty$
with  $N/R \to 2$ 
\begin{equation} \label{eq:wRapprox} 
	\lim_{R \to \infty} - \frac{1}{N} \log w_R(t) = NV(t), \qquad t \in (-1,1). 
	\end{equation}
The limit \eqref{eq:wRapprox} follows from the inequalities given in Proposition \ref{prop1} below. 

In what follows we assume that $R, N$ are given in such a way that 
$N/R \to 2$ as $R \to \infty$, and we may write $\lim_{R \to \infty}$
and $\lim_{N\to\infty}$ interchangeably.
For every $R> 0$ we define
\begin{equation} \label{eq:epsR}
	\varepsilon_R = \frac{2R}{N} \sum_{|p_n| > R} \frac{1}{p_n}.
	\end{equation}
Because of the assumption \eqref{convergence:pv} and \eqref{eq:NRlimit}
we have $\varepsilon_R \to 0$ as $R \to \infty$. 

The following proposition will be proven in Section \ref{sec:proofprop1}.
\begin{proposition}  \label{prop1}
\begin{enumerate}
\item[(a)] For every $\alpha > 1$, there is $R_{\alpha} > 0$ such that if $R \geq R_{\alpha}$, then  
\begin{equation} \label{eq:wRupperest} 
	w_R(t)  \leq 
	e^{-N \left(V \left( \frac{t}{\alpha}\right) + \varepsilon_R t \right)}, 
	\qquad t \in [-1,1].
\end{equation}
\item[(b)] For every $\alpha>1$ and $\beta\in (0,1)$, there is $R_{\alpha,\beta}>0$ such that if $R \geq R_{\alpha,\beta}$, then
\begin{equation} \label{eq:wRlowerest}
	w_R(t)  \geq e^{-N(V(\alpha t) + \varepsilon_R t)}, 
	\qquad t \in [-\beta, \beta].
	\end{equation}
\end{enumerate}
\end{proposition}
Note that equality holds in \eqref{eq:wRupperest} and \eqref{eq:wRlowerest}
at $t=0$. The number \eqref{eq:epsR} is such that the derivatives
at $t=0$ are the same as well, which is of course necessary 
in order to have the inequalities \eqref{eq:wRupperest} and \eqref{eq:wRlowerest}
with equality at $t=0$, since $V(0) = V'(0) = 0$.

Based on \eqref{eq:wRupperest} and \eqref{eq:wRlowerest} we
introduce for every $R > 0$ and every $\alpha>1$ 
the following two weights on $[-1,1]$:
\begin{align} \label{eq:wRplus}
	w_{R,\alpha}^+(t) & = 
		(1-t^2)^{-1/2} e^{-N \left(V \left( \frac{t}{\alpha}\right) + 
		\varepsilon_R t \right)},  \\
		\label{eq:wRminus}
	w_{R,\alpha}^-(t) & = (1-\beta^{-2} t^2)^{1/2}
			e^{-N(V(\alpha t) + \varepsilon_R t)}
			\chi_{[-\beta, \beta]}(t),\quad \beta=\alpha^{-2},
	\end{align}
where we introduced the square root prefactors for convenience of further 
asymptotic analysis.

The weights $w_{R,\alpha}^{\pm}$ are sufficiently nice, and 
with known techniques we can find the universality limits for them.
This is stated in part (b) of the next proposition. Part (a) follows
almost immediately from Propositon \ref{prop1} and the definitions. 

\begin{proposition} \label{prop2}
\begin{enumerate}
\item[(a)] For each $\alpha > 1$, there is a number $R_{\alpha}$
such that  for $R \geq R_{\alpha}$,
\begin{equation} \label{eq:ass1a} 
	w_{R,\alpha}^-(t) \leq w_R(t) \leq w_{R,\alpha}^+(t), 
	\qquad t \in [-1,1],
\end{equation}
and
\begin{equation} \label{eq:ass1b} 
	\lim_{R \to \infty} w_{R,\alpha}^-\left(\frac{x}{R}\right) = 
	\lim_{R \to \infty} w_{R,\alpha}^+\left(\frac{x}{R}\right) = 1 
\end{equation}
uniformly for $x$ in compact sets.
\item[(b)] There exist constants $c_\alpha^{\pm} > 0$ such that
for every $\alpha > 1$,
\begin{equation} \label{eq:ass2a} 
	\lim_{R \to \infty} 
	\frac{1}{N} K_N\left( \frac{x}{N}, \frac{y}{N}; w_{R,\alpha}^{\pm}\right) 
		= \frac{\sin \pi c_{\alpha}^{\pm} (x-y)}{\pi(x-y)}
		\end{equation}
uniformly for $x$ and $y$ in compact sets, and the  constants $c_\alpha^{\pm}$
 satisfy
\begin{equation} \label{eq:ass2b} 
	\lim_{\alpha \to 1+} c_{\alpha}^{-} = 
	\lim_{\alpha \to 1+} c_{\alpha}^{+} = \frac{1}{2}. 
	\end{equation}
\end{enumerate}
\end{proposition}
Proposition \ref{prop2} is proved in Section \ref{sec:proofprop2}.

In Section \ref{sec:proofmain} we turn to the proof of our main Theorem \ref{mainthm:wR}.
We use ideas of Lubinsky \cite{LubinskyAnnals} in order to compare 
the kernel $K_N(x,y; w_R)$
with those for the weights $w_{R, \alpha}^{\pm}$. The properties
listed in Proposition \ref{prop2} turn out to be enough to 
prove Theorem \ref{mainthm:wR}, and this in turn gives us
Theorem \ref{mainthm} as we already mentioned.

\section{Proof of Proposition \ref{prop1}} \label{sec:proofprop1}

\subsection{A lemma}
We start the proof of Proposition \ref{prop1} with a lemma. 

\begin{lemma}\label{weights:inequality} 
Suppose $R$ is a positive integer. Then,
\begin{equation}  \label{eq:weightsinequality} 
\prod_{n  = R}^{\infty} 
\left( 1 - \frac{R^2 t^2}{n^2} \right)^2 \leq e^{-2RV(t)}
	\leq \prod_{n=R+1}^{\infty} 
\left( 1 - \frac{R^2 t^2}{n^2} \right)^2 , \qquad t \in [-1,1]. 
\end{equation}
\end{lemma}
\begin{proof}
Let $t \in [-1,1]$. The function $x \mapsto \psi(x) = 
-\log\left(1- \frac{R^2t^2}{x^2}\right)$ is positive and decreasing for $x \geq R$. 
Therefore, by comparing the integral with Riemann sums, we have
\[ \sum_{n=R+1}^{\infty} \psi(n) \leq  
	 \int_R^{\infty} \psi(x) dx \leq \sum_{n=R}^{\infty} \psi(n), \]
since $R$ is an integer. Since $\int_R^{\infty} \psi(x) dx = RV(t)$, we get
\[  - \sum_{n=R}^{\infty} \psi(n) \leq -RV(t)
	\leq   - \sum_{n=R+1}^{\infty} \psi(n). \]
The inequality \eqref{eq:weightsinequality} follows by taking 
exponentials and squaring the expressions.
\end{proof}

\subsection{Proof of Proposition \ref{prop1} (a)}

\begin{proof}
Take $\alpha > 1$. Since and $2R/N \to 1$ as $R \to \infty$
and $p_n/n \to 1$ as $n \to \infty$, there is $R_{\alpha}$ such that 
for $R \geq R_{\alpha}$ we have
\begin{equation} \label{eq:Ralpha1}
	 \frac{2R}{\alpha} < N < 2 \alpha R,
\end{equation}
and if $|n| \geq  \lfloor \alpha R \rfloor $, then
	\begin{equation} \label{eq:Ralpha2}
		\frac{1}{\alpha} \leq \frac{p_n}{n} \leq \alpha,
		\end{equation}
	and
	\begin{equation} \label{eq:Ralpha3}
		|p_n| > R.
		\end{equation}

We then write for $R \geq R_{\alpha}$ and $t \in [-1,1]$,
\begin{align} \nonumber 
w_R(t) & = \prod_{|p_n| \geq R} \left(1- \frac{Rt}{p_n} \right)^2 \\
	& 
	= \prod_{n \in S_1}  \left(1- \frac{Rt}{p_n} \right)^2 \,
	 \left(\prod_{n \in S_2 \cup S_3}  \left(1- \frac{Rt}{p_n} \right)^2  
		\left(1- \frac{Rt}{p_{-n}} \right)^2 \right)  \label{eq:wRsplit}
		\end{align}
where the sets $S_1$, $S_2$ and $S_3$ are defined as follows
\begin{equation} \label{eq:setSj}
\begin{aligned}
	S_1 & = \{ n \in \mathbb Z \mid |p_n| > R \text{ and }
		|n| < \lfloor \alpha R \rfloor \}, \\ 
	S_2 & = \{ n \in \mathbb N \mid n \geq \lfloor \alpha R \rfloor 
		\text{ and }
		p_n \geq -p_{-n} \}, \\
	S_3 & = \{ n \in \mathbb N \mid n \geq \lfloor \alpha R \rfloor 
		\text{ and }
		p_n < -p_{-n} \}. 
\end{aligned}
\end{equation} 
Because of \eqref{eq:Ralpha3}
and the above definitions \eqref{eq:setSj}, we have a disjoint union
\begin{equation} \label{eq:Sunion} 
	\{ n \mid |p_n| > R \}
	= S_1 \cup S_2 \cup (-S_2) \cup S_3 \cup (-S_3), 
	\end{equation}
which is used in the factorization \eqref{eq:wRsplit}.
	Also $S_1$ is a finite set.

For $n \in S_1$, we use the estimate 
$1 - \frac{Rt}{p_n}  \leq e^{- \frac{Rt}{p_n}}$
which follows from the elementary inequalities $0 \leq 1+x \leq e^x$  if
$x \geq -1$. Thus for $t \in [-1,1]$,
\begin{equation} \label{eq:S1prod} 
	\prod_{n \in S_1}  \left(1 - \frac{Rt}{p_n} \right)^2
	\leq e^{-2Rt \sum\limits_{n\in S_1} \frac{1}{p_n}}. 
	\end{equation}
	
For $n \in S_2$, we write for $t \in [-1,1]$,
\begin{align} \nonumber 
	\left(1 - \frac{Rt}{p_n} \right) \left(1 - \frac{Rt}{p_{-n}} \right)
	& =  
	\left(1 - \frac{Rt}{p_n} \right) 	\left(1 + \frac{Rt}{p_n} - 
	Rt \left(\frac{1}{p_n} +  \frac{1}{p_{-n}} \right) \right) \\
	& = \label{eq:S2prod1} 
	\left(1 - \frac{R^2t^2}{p_n^2} \right)
	\left(1 + \frac{Rt}{1+ \frac{Rt}{p_n}} \left(-\frac{1}{p_n}- \frac{1}{p_{-n}}\right)\right).
\end{align}
Since $p_n \geq - p_{-n} \geq R$ by  \eqref{eq:Ralpha3} and the definition 
of $S_2$ in \eqref{eq:setSj}, we have 
$-\frac{1}{p_n}- \frac{1}{p_{-n}} \geq 0$, and
we can use $\frac{Rt}{1+ \frac{Rt}{p_n}} \leq Rt$ for $t \in [-1,1]$, 
to obtain from \eqref{eq:S2prod1}
\begin{align} \nonumber 
	\left(1 - \frac{Rt}{p_n} \right)\left(1 - \frac{Rt}{p_{-n}} \right)
	& \leq 
	\left(1 - \frac{R^2t^2}{p_n^2} \right)
	\left(1 + Rt \left(-\frac{1}{p_n}- \frac{1}{p_{-n}}	\right) \right) 
	\\ \label{eq:S2prod2} 
	& \leq  	\left(1 - \frac{R^2t^2}{p_n^2} \right)
	e^{-Rt \left(\frac{1}{p_n} + \frac{1}{p_{-n}} \right)},
	\end{align}
where we used again that $1+x \leq e^x$. Since $p_n \leq n \alpha$
by \eqref{eq:Ralpha2}, we can further estimate this by
\begin{equation}  \label{eq:S2prod3} 
	\left(1 - \frac{Rt}{p_n} \right)\left(1 - \frac{Rt}{p_{-n}} \right)
	 \leq  	\left(1 - \frac{R^2t^2}{\alpha^2 n^2} \right)
	e^{-Rt \left(\frac{1}{p_n} + \frac{1}{p_{-n}}\right)}.
	\end{equation}
	
The estimates for $n \in S_3$ are similar. We interchange the
roles of $p_n$ and $p_{-n}$ in \eqref{eq:S2prod1}, we use 
$\frac{Rt}{1+\frac{Rt}{p_{-n}}} \geq Rt$ for $t \in [-1,1]$ to
obtain \eqref{eq:S2prod2} with $p_n$ interchanged with $p_{-n}$. 
Since $0 > p_{-n} \geq -n \alpha$, we also find \eqref{eq:S2prod3} 
for $n \in S_3$. Hence  
\begin{equation} \label{eq:S2S3prod} 
	\prod_{n \in S_2 \cup S_3}  \left(1 - \frac{Rt}{p_n} \right)^2 
		\left(1- \frac{Rt}{p_{-n}} \right)^2 
	\leq  	\prod_{n \in S_2 \cup S_3}
	\left(1 - \frac{R^2t^2}{\alpha^2 n^2} \right)^2
	e^{-2Rt \sum\limits_{n \in S_2 \cup S_3} \left(\frac{1}{p_n} + \frac{1}{p_{-n}}\right)}.
	\end{equation}
	
Using \eqref{eq:S1prod} and \eqref{eq:S2S3prod}
in \eqref{eq:wRsplit}, we obtain
\begin{align} 
	w_R(t)  & \leq  \nonumber
	\left(	\prod_{n \in S_2 \cup S_3}
	\left(1 - \frac{R^2t^2}{\alpha^2 n^2} \right)^2 \right)
		e^{-2Rt \left( \sum\limits_{n \in S_1} \frac{1}{p_n} + 
		\sum\limits_{n \in S_2 \cup S_3} 
		\left(\frac{1}{p_n} + \frac{1}{p_{-n}}\right)
		 \right)} \\
	 & = \label{eq:wRplus2inequality} 
	\left( \prod_{n = \lfloor \alpha R \rfloor}^{\infty} 
	\left(1 - \frac{R^2t^2}{\alpha^2 n^2} \right)^2 \right)
	 e^{-N \varepsilon_R t},
	 \end{align}
where in the last step we used \eqref{eq:epsR} and the 
property \eqref{eq:Sunion}.

Then we use the first inequality in Lemma \ref{weights:inequality}
with $\lfloor \alpha R \rfloor$ instead of $R$ and $t/\alpha$ instead of $t$.
Then by \eqref{eq:weightsinequality} and \eqref{eq:wRplus2inequality},
\[ w_R(t) \leq e^{-2 \lfloor \alpha R \rfloor V \left( \frac{t}{\alpha} \right)}
	e^{-N \varepsilon_R t}. \]
We finally use \eqref{eq:Ralpha1} and we obtain the  required
inequality \eqref{eq:wRupperest}.   \end{proof}

\subsection{Proof of Proposition \ref{prop1} (b)}

\begin{proof}
The proof for the lower estimate \eqref{eq:wRlowerest} follows along the same lines, 
but it is more involved.
Instead of the inequality $1+x \leq e^x$, we now use a less obvious,
but still elementary, inequality
\begin{equation} \label{eq:1pxlower0} 
	1+ x \geq e^{x- \frac{x^2}{1+x}}, \qquad x > - 1. 
	\end{equation}
It follows from \eqref{eq:1pxlower0}
that for any $\beta \in (0,1)$,
\begin{equation} \label{eq:1pxlower} 
	1+ x \geq e^{x-  \frac{1}{1-\beta} x^2}, \qquad x \geq  - \beta. 
	\end{equation}

Then, fix $\alpha > 1$, $\beta\in (0,1)$,  and take $\gamma\in (1,2)$ so close to $1$ that
\begin{equation} \label{eq:gamma}
	\gamma^2 + 12 \frac{\gamma^2 - 1}{1-\beta} < \alpha^2. 
	\end{equation}

As in the proof of the upper estimate we assume that $R$ is large enough
so that \eqref{eq:Ralpha1},  \eqref{eq:Ralpha2}, \eqref{eq:Ralpha3} hold,
but with $\alpha$ replaced by $\gamma$. Thus we take 
$R_{\gamma} > \frac{4}{\gamma^2-1}$  big enough
such that for $R \geq R_{\gamma}$ we have
\begin{equation} \label{eq:Rgamma1}
	 \frac{2R}{\gamma} < N < 2 \gamma R,
\end{equation}
and if $|n| \geq  \lfloor \gamma R \rfloor $, then
	\begin{equation} \label{eq:Rgamma2}
		\frac{1}{\gamma} \leq \frac{p_n}{n} \leq \gamma,
		\end{equation}
	and
	\begin{equation} \label{eq:Rgamma3}
		|p_n| > R.
		\end{equation}
Since $R_{\gamma} > \frac{4}{\gamma^2-1}$, we also have that
for $R \geq R_{\gamma}$ 
\begin{equation} \label{eq:Rgamma4} 
	R \geq \frac{4}{\gamma^2-1}.
	\end{equation}

We fix $R \geq R_{\gamma}$ and $t \in [-\beta, \beta]$. 
We use the 
factorization of $w_R$ as in \eqref{eq:wRsplit}, but with
slightly different sets $S_1, S_2, S_3$, namely
\begin{equation} \label{eq:setSjstar}
\begin{aligned}
	S_1^* & = \{ n \in \mathbb Z \mid |p_n| > R \text{ and }
		|n| \leq \lceil \gamma R \rceil \}, \\ 
	S_2^* & = \{ n \in \mathbb N \mid n \geq \lceil \gamma R \rceil + 1 
		\text{ and }
		p_n \geq -p_{-n} \}, \\
	S_3^* & = \{ n \in \mathbb N \mid n \geq \lceil \gamma R \rceil + 1 
		\text{ and }
		p_n < -p_{-n} \}. 
\end{aligned}
\end{equation} 

For $n \in S_1^*$ we obtain from \eqref{eq:1pxlower}
that for $t\in[-\beta,\beta]$
\[ 1 - \frac{Rt}{p_n} \geq e^{- \frac{Rt}{p_n} - 
	\frac{(Rt)^2}{(1-\beta) p_n^2}}\,, \]
since $|\frac{Rt}{p_n}| \leq |t| \leq \beta$, and therefore
\begin{equation} \label{eq:S1prodlow}
 \prod_{n \in S_1^*} \left(1 - \frac{Rt}{p_n} \right)^2
 	\geq e^{- 2Rt \sum\limits_{n \in S_1} \frac{1}{p_n}
 	- 2 \frac{R^2 t^2}{1-\beta} 
 	\sum\limits_{n \in S_1} \frac{1}{p_n^2}}.
\end{equation}

For $n \in S_2^*$, we do not use \eqref{eq:S2prod1} but rather
\begin{align} \label{eq:S2prod1low}  
	\left(1 - \frac{Rt}{p_n} \right) \left(1 - \frac{Rt}{p_{-n}} \right)
	= \left(1 - \frac{R^2t^2}{p_{-n}^2} \right)
	\left(1 + \frac{Rt}{1+ \frac{Rt}{p_{-n}}} 
	\left(-\frac{1}{p_n}- \frac{1}{p_{-n}}\right) 	\right).
\end{align}
Since $p_{-n} \leq - R$, see \eqref{eq:Rgamma3}, 
we have $\frac{Rt}{1+ \frac{Rt}{p_{-n}}} \geq Rt$,
and we obtain since $-\frac{1}{p_n}- \frac{1}{p_{-n}} \geq 0$,
\begin{align} \label{eq:S2prod2low}  
	\left(1 - \frac{Rt}{p_n} \right) \left(1 - \frac{Rt}{p_{-n}} \right)
	\geq \left(1 - \frac{R^2t^2}{p_{-n}^2} \right)
	\left(1 + Rt \left(-\frac{1}{p_n}- \frac{1}{p_{-n}} \right)\right).
\end{align}
Since $p_n \geq - p_{-n} \geq R$ we then have
\[ \left| Rt \left(-\frac{1}{p_n}- \frac{1}{p_{-n}} \right)\right|
	\leq \left| \frac{Rt}{p_{-n}} \right| \leq |t| \leq \beta,
	\]	
and therefore we can apply  \eqref{eq:1pxlower} to estimate
the last factor in \eqref{eq:S2prod2low}. In the first factor
we use  $p_{-n}^2 \geq  \left(\frac{n}{\gamma} \right)^2$, see \eqref{eq:Rgamma2},
and we obtain from \eqref{eq:S2prod2low} 
\begin{align} \label{eq:S2prod3low}  
	\left(1 - \frac{Rt}{p_n} \right) \left(1 - \frac{Rt}{p_{-n}} \right)
	\geq \left(1 - \frac{\gamma^2 R^2t^2}{n^2} \right)
	e^{-Rt \left( \frac{1}{p_n} + \frac{1}{p_{-n}} \right)
	- \frac{R^2 t^2}{1-\beta} R^2 t^2 
	\left( \frac{1}{p_n} + \frac{1}{p_{-n}} \right)^2}.
\end{align}

The same estimate \eqref{eq:S2prod3low} holds  
for $n \in S_3^*$ as well (we find it by interchanging the roles of 
$p_n$ and $p_{-n}$ in \eqref{eq:S2prod2low}), and we obtain
\begin{multline} \label{eq:S2S3prodlow}
	\prod_{n \in S_2^* \cup S_3^*} 
	\left(1 - \frac{Rt}{p_n} \right) \left(1 - \frac{Rt}{p_{-n}} \right)^2 
	\geq \prod_{n \in S_2^* \cup S_3^*} 
	\left(1 - \frac{\gamma^2 R^2t^2}{n^2} \right)^2 \\
	\times
	e^{-2Rt  \sum\limits_{n \in S_2^* \cup S_3^*} 
		\left( \frac{1}{p_n} + \frac{1}{p_{-n}} \right)
		- 2\frac{R^2t^2}{1-\beta}  
	  \sum\limits_{n \in S_2^* \cup S_3^*} 
	  \left( \frac{1}{p_n} + \frac{1}{p_{-n}} \right)^2}.
\end{multline}

Combining \eqref{eq:wRsplit} (with $S_j$ replaced by $S_j^*$ for $j=1,2,3$),
\eqref{eq:S1prodlow} and \eqref{eq:S2S3prodlow} we have for 
$t \in [-\beta, \beta]$,
\begin{align} \nonumber
	w_R(t) & \geq \left(
	\prod_{n \in S_2^* \cup S_3^*} 
	\left(1 - \frac{\gamma^2 R^2t^2}{n^2} \right)^2	\right)
	e^{-2Rt \left(\sum\limits_{n \in S_1^*} \frac{1}{p_n} +  
		\sum\limits_{n \in S_2^* \cup S_3^*} 
		\left( \frac{1}{p_n} + \frac{1}{p_{-n}} \right) \right)} \\
		\nonumber
		& \qquad \times
		e^{- \frac{2R^2}{1-\beta}  
		\left(\sum\limits_{n\in S_1^*} \frac{1}{p_n^2}
		+ \sum\limits_{n \in S_2^* \cup S_3^*} 
	  \left( \frac{1}{p_n} + \frac{1}{p_{-n}} \right)^2\right) t^2} \\
	  \nonumber 	 & = 
	\left(	\prod_{n = \lceil \gamma R \rceil + 1}^{\infty} 
	\left(1 - \frac{\gamma^2 R^2t^2}{n^2} \right)^2	 \right)
	e^{-N  \varepsilon_R t} \\
	&  \qquad \times \label{eq:wRlower2}
		e^{- \frac{2R^2}{1-\beta}  
		\left(\sum\limits_{n\in S_1^*} \frac{1}{p_n^2}
		+ \sum\limits_{n \in S_2^* \cup S_3^*} 
	  \left( \frac{1}{p_n} + \frac{1}{p_{-n}} \right)^2\right) t^2}.
\end{align}
where we used \eqref{eq:epsR}.

It remains to estimate the first and third factor in the
product in the right-hand side of\eqref{eq:wRlower2}. 

To estimate the first factor  we use the second inequality
of \eqref{eq:weightsinequality} with $\lceil \gamma R \rceil$
instead of $R$ and $\frac{\gamma R}{\lceil \gamma R \rceil} t$ instead
of $t$. This gives us
\begin{equation} \label{eq:claim1a} 
	\prod_{n = \lceil \gamma R \rceil + 1}^{\infty} 
	\left(1 - \frac{\gamma^2 R^2t^2}{n^2} \right)^2	 
	\geq e^{-2 \lceil \gamma R \rceil V(\frac{\gamma R}{\lceil \gamma R \rceil}t)}  
\end{equation}
Since $V$ is convex  and $V(0) = 0$ (as can easily 
be checked from the definition \eqref{eq:Vdef}), we have
\[ V\left(\frac{\gamma Rt}{\lceil \gamma R \rceil}\right) \leq 
	\frac{\gamma R}{\lceil \gamma R \rceil}  V(t). \]
Using this in \eqref{eq:claim1a}, and also recalling $2R \leq \gamma N$,
see \eqref{eq:Rgamma1}, we obtain   
\begin{equation} \label{eq:claim1}
	\prod_{n = \lceil \gamma R \rceil + 1}^{\infty} 
	\left(1 - \frac{\gamma^2 R^2t^2}{n^2} \right)^2	 
	\geq e^{- \gamma^2 N V(t)}.
\end{equation}

Next we estimate the third factor in \eqref{eq:wRlower2}.
If $n \in S_1^*$ then $|p_n| > R$ and 
\begin{equation} \label{eq:claim2a}
	\sum_{n \in S_1^*} \frac{1}{p_n^2} \leq \frac{\# S_1^*}{R^2}. 
	\end{equation}
Also if $n \in S_1^*$ we have $|n| \leq \lceil \gamma R \rceil$
by the definition of $S_1^*$ in \eqref{eq:setSjstar}, 
and $|n| \geq |p_n|/\gamma \geq R/\gamma$, where we also used 
\eqref{eq:Rgamma2}. Thus $S_1^*$ is contained in $[-\gamma R -1, - \frac{R}{\gamma}] \cup 
	[\frac{R}{\gamma}, \gamma R+1]$.
Since $S_1^*$  contains integers only, we obtain 
\[ \# S_1^* \leq 2(\gamma R + 1 - \tfrac{R}{\gamma} +1)  
	< 2(\gamma^2-1) R + 4. \]
Combining this with \eqref{eq:claim2a} and \eqref{eq:Rgamma4} we find
\begin{equation} \label{eq:claim2b}
	\sum_{n \in S_1^*} \frac{1}{p_n^2} \leq 
		\frac{2 (\gamma^2-1)}{R} + \frac{4}{R^2} \leq \frac{3(\gamma^2-1)}{R}. 
	\end{equation}

If $n \in S_2^* \cup S_3^*$ then $n \geq \lceil \gamma R \rceil + 1$ 
by \eqref{eq:setSjstar}, and both $n/p_n$ and $-n/p_{-n}$ lie between 
$1/\gamma$ and $\gamma$ by \eqref{eq:Rgamma2}. 
Then
\[ \left|\frac{1}{p_n} + \frac{1}{p_{-n}} \right| \leq 
	\frac{\gamma - 1/\gamma}{n}  < \frac{\gamma^2-1}{n} \]
and therefore
\begin{align} \nonumber 
	\sum\limits_{n \in S_2^* \cup S_3^*} 
	  \left( \frac{1}{p_n} + \frac{1}{p_{-n}} \right)^2 
	 & \leq \sum_{n = \lceil \gamma R \rceil + 1}^{\infty}  
	 \left(\frac{\gamma^2-1}{n} \right)^2 \\ \nonumber
	 & \leq (\gamma^2-1)^2 \int_{\gamma R}^{\infty} \frac{dx}{x^2} \\
	 & \leq \frac{(\gamma^2-1)^2}{R} \nonumber \\
	 & \leq \frac{3(\gamma^2-1)}{R}. \label{eq:claim2c}
\end{align}
In the last step we used $\gamma^2-1 <3$.
Using \eqref{eq:claim2b} and \eqref{eq:claim2c}, we find the
inequality
\[ \sum\limits_{n\in S_1^*} \frac{1}{p_n^2}
		+ \sum\limits_{n \in S_2^* \cup S_3^*} 
	  \left( \frac{1}{p_n} + \frac{1}{p_{-n}} \right)^2
	  \leq \frac{6(\gamma^2-1)}{R}. \]
Since $2R \leq \gamma N$ with $\gamma <2$, we have
\[ 2R^2 \frac{6 (\gamma^2-1)}{R}
	\leq 12N (\gamma^2-1), \]
and therefore 
\begin{equation} \label{eq:claim2} 
	\frac{2R^2}{1-\beta}  
		\left(\sum\limits_{n\in S_1^*} \frac{1}{p_n^2}
		+ \sum\limits_{n \in S_2^* \cup S_3^*} 
	  \left( \frac{1}{p_n} + \frac{1}{p_{-n}} \right)^2\right)
	  \leq 12 N \frac{\gamma^2-1}{1-\beta}.
	  \end{equation}

Having \eqref{eq:claim1} and \eqref{eq:claim2}
we can continue with our lower bound for $w_R(t)$. From \eqref{eq:wRlower2} 
we obtain
\begin{equation}
 \label{eq:wRlower3}
 	w_R(t) \geq e^{-\gamma^2 N V(t) - 12N \frac{\gamma^2-1}{1-\beta} t^2}
 	e^{-N \varepsilon_R t}.
\end{equation}

From the definition of $V$ in \eqref{eq:Vdef}, we find the Maclaurin series
\begin{equation} \label{eq:Vseries} 
	V(t) = t^2 + \sum_{k=2}^{\infty} \frac{t^{2k}}{k(2k-1)}
	\end{equation}
from which it is clear that $t^2 \leq V(t)$. Thus by \eqref{eq:gamma},
\begin{align*} 
	\gamma^2 N V(t) + 12N \frac{\gamma^2-1}{1-\beta} t^2
	& \leq \left(\gamma^2 + 12 \frac{\gamma^2-1}{1-\beta} \right) N V(t) 
	  \leq \alpha^2 N V(t).
	 \end{align*}
From the Maclaurin series \eqref{eq:Vseries} it is also clear that 
$\alpha^2 V(t) \leq V(\alpha t)$, since $\alpha > 1$. Hence
\[	\gamma^2 N V(t) + 12N \frac{\gamma^2-1}{1-\beta} t^2 \leq N V(\alpha t), \]
and using this in \eqref{eq:wRlower3}, we 
obtain the required inequality \eqref{eq:wRlowerest}. 
\end{proof}

\section{Proof of Proposition \ref{prop2}} \label{sec:proofprop2}

\subsection{Proof of Proposition \ref{prop2} (a)}

\begin{proof}
The inequalities in \eqref{eq:ass1a} follow from Proposition \ref{prop1} and
the definitions \eqref{eq:wRplus}-\eqref{eq:wRminus}.

The limits in \eqref{eq:ass1b} are also almost immediate. From \eqref{eq:wRplus}
we have
\begin{equation} \label{eq:wRpluslimit} 
	w_{R,\alpha}^+ \left(\frac{x}{R}\right)
	= \left(1- \frac{x^2}{R^2}\right)^{-1/2}
		\exp\left( -N  V\left( \frac{x}{\alpha R} \right) 
			- \varepsilon_R \frac{Nx}{R} \right). 
			\end{equation}
Clearly $\left(1- \frac{x^2}{R^2}\right)^{-1/2} \to 1$ as $R \to \infty$,
and by \eqref{eq:epsR}
\[ \varepsilon_R \frac{Nx}{R} = 2 x\sum_{|p_n| > R} \frac{1}{p_n}, \]
which tends to $0$ as $R \to \infty$. 

From  \eqref{eq:Vseries} we see that 
$V(t) = O(t^2)$ as $t \to 0$ and therefore 
\[ N V\left(\frac{x}{\alpha R}\right) = N O\left(\left(\frac{x}{\alpha R}\right)^2\right) 
= \frac{x^2}{\alpha^2} O\left( \frac{N}{R^2} \right) \]
and this also tends to $0$ as $R \to \infty$.
Thus \eqref{eq:wRpluslimit} tends to $1$ as $R \to \infty$, and
the convergence is uniform for $x$ in a compact set.
This proves the second limit in \eqref{eq:ass1b}. The proof for
the first limit is the same.
\end{proof}

\subsection{Preliminaries to the proof of Proposition \ref{prop2} (b)}

To establish the limits \eqref{eq:ass2a} for the weights 
$w^{\pm}_{R,\alpha}$ we need a few concepts of logarithmic potential theory with 
external fields \cite{SaffTotik}. 

Let $W$ be a continuous function on a compact interval $I$ of the real line. 
Among all probability measures $\mu$ with $\supp(\mu)\subset I$, there exists a 
unique measure $\mu_W$, called the equilibrium measure in the presence of the 
external field $W$, which minimizes the weighted energy
\[
	\iint\log\frac{1}{|x-t|}d\mu(x)d\mu(t)+\int W(t)d\mu(t).
\] 
If we denote by $U^\mu$ the logarithmic potential of a measure $\mu$, that is,
\begin{equation} \label{eq:Umu}
	U^\mu(x)=\int\log\frac{1}{|x-t|}d\mu(t),
\end{equation}
then $\mu_W$  is characterized by the property that 
\begin{equation} \label{eq:Umuconditions}
\begin{aligned}
	2 U^{\mu_W}(x)+ W(x) \geq{} & \ell, \qquad x\in I\setminus\supp(\mu_W),\\
	2 U^{\mu_W}(x)+ W(x) = {} &\ell,\qquad x\in\supp(\mu_W),
\end{aligned}
\end{equation}
for some constant $\ell$ that depends on $W$.

We will specifically consider external fields of the form 
\begin{align}\label{eq:Valphaeps}
V_{\alpha, \varepsilon}(x) = V\left(\frac{x}{\alpha}\right)+\varepsilon x,
\end{align}
where $\alpha\geq 1$ and $\varepsilon$ are real numbers, and $V$ is given by 
\eqref{eq:Vdef} as before. We use $\psi_{\alpha, \varepsilon}$ to denote
the density of the equilibrium measure with external field \eqref{eq:Valphaeps}
on the interval $[-1,1]$. We already noted in \eqref{eq:psiV} that
\begin{equation} \label{eq:psi10} 
	\psi_{1,0}(x) = \frac{1}{2}, \qquad x \in [-1,1]. 
	\end{equation}

For $\alpha > 1$ and $\varepsilon$ close to zero we can also calculate the
density explicitly. We write
\begin{equation} \label{eq:epsalpha} 
	\varepsilon_{\alpha} := 2 \sqrt{1-\alpha^{-2}}. 
	\end{equation}

\begin{proposition}\label{form:equilibrium-measures}
Let $\alpha>1$ and $\varepsilon \in [-\varepsilon_{\alpha}, \varepsilon_{\alpha}]$.
Then the equilibrium measure $\mu_{V_{\alpha,\varepsilon}}$ in the presence
of the external field \eqref{eq:Valphaeps} has full support $[-1,1]$ with density given
by
\begin{equation} \label{eq:psialphaeps} 	
	\psi_{\alpha, \varepsilon}(x) = 
	\frac{2\sqrt{\alpha^2-1} -\alpha \varepsilon x}{2\alpha\pi \sqrt{1-x^2}}
			+ \frac{1}{\alpha \pi} \arctan\left( \frac{\sqrt{1-x^2}}{\sqrt{\alpha^2-1}} \right),
			\qquad x \in (-1,1).	
\end{equation}
\end{proposition}
\begin{proof}
Note that the density \eqref{eq:psialphaeps} is indeed positive on $(-1,1)$
because of the assumption $|\varepsilon| \leq \varepsilon_{\alpha}$, see \eqref{eq:epsalpha}.
The density is zero at one of the endpoints in case  $|\varepsilon| = \varepsilon_{\alpha}$.

We first consider $\varepsilon = 0$. 

The equilibrium measure for the external field $V(x/\alpha)$ on the bigger
interval $[-\alpha, \alpha]$ is a multiple of the Lebesgue measure, namely $\frac{1}{2\alpha} dx$
restricted to $[-\alpha, \alpha]$. This simply follows from \eqref{eq:psi10}  by rescaling.
The equilibrium measure $\mu_{V_{\alpha,0}}$ is then the balayage of 
$\frac{1}{2\alpha} dx$ onto  $[-1,1]$.
The balayage of a point mass $\delta_t$ onto $[-1,1]$ is the measure with density
\[ 
\frac{1}{\pi \sqrt{1-x^2}} \frac{\sqrt{t^2-1}}{|t-x|} dx,
		\qquad \text{for } t > 1 \text{ or } t < -1. \]
It follows that
\begin{align*} 
	\psi_{\alpha,0}(x) & = \frac{1}{2\alpha} + \frac{1}{2\alpha\pi \sqrt{1-x^2}}
	\left(\int_{-\alpha}^{-1} + \int_1^{\alpha} \right) \frac{\sqrt{t^2-1}}{|t-x|} dt, \\
	& = \frac{1}{2\alpha}  + \frac{1}{\alpha\pi \sqrt{1-x^2}}
	\int_1^\alpha   \frac{t \sqrt{t^2-1} }{t^2-x^2} \, dt, \qquad x \in [-1,1]. 
	\end{align*}
A change of variables $t^2 = s^2 + 1$ leads to
\begin{align*} 
	\psi_{\alpha,0}(x) = \frac{1}{2\alpha}  + \frac{1}{\alpha\pi \sqrt{1-x^2}}
	\int_0^{\sqrt{\alpha^2-1}}   \frac{s^2 }{s^2 + 1 -x^2} ds, \qquad x \in [-1,1].  
	\end{align*}
	The integral can be evaluated, and the result is
\begin{equation} \label{eq:psialpha0} 
	\psi_{\alpha,0}(x) = \frac{1}{2\alpha}+ \frac{1}{\alpha\pi \sqrt{1-x^2}}
	\left[ \sqrt{\alpha^2-1} - 
	\sqrt{1-x^2} \arctan\left(\frac{\sqrt{\alpha^2-1}}{\sqrt{1-x^2}} \right) \right],
	\end{equation}
for $x \in [-1,1]$. Using $\arctan(y) + \arctan(1/y)  = \pi/2$, we see that \eqref{eq:psialpha0}
leads to \eqref{eq:psialphaeps} for the case $\varepsilon=0$.

Let now $\nu$ be the signed measure on $[-1,1]$ given by
\begin{equation} \label{eq:nu}
	d\nu=\frac{ x}{\pi \sqrt{1-x^2}}dx.
\end{equation}
We claim that
\begin{equation} \label{eq:Unu}
	U^{\nu}(x) = x, \qquad x \in [-1,1].
\end{equation}

To prove \eqref{eq:Unu} we recall that the measure $\omega$ with density
\[ \frac{d\omega}{dx} = \frac{1}{\pi\sqrt{1-x^2}}, \qquad x \in [-1,1] \]
is the equilibrium measure of the interval $[-1,1]$. Thus $U^{\omega}$
is constant on $[-1,1]$ and therefore $\frac{d}{dx} U^{\omega}(x) = 0$ for
$x \in (-1,1)$. This means
\[ 	 -\frac{1}{\pi} \dashint_{-1}^1 \frac{1}{x-t }\frac{1}{\sqrt{1-t^2}}dt=0,
 	 	\qquad x \in [-1,1],
\]
where $\dashint$ denotes the Cauchy principal value.
Hence
\begin{align*}
	\frac{d}{dx} U^\nu(x) & = 
	-\frac{1}{\pi}\dashint_{-1}^{1}\frac{1}{x-t}\frac{t}{\sqrt{1-t^2}}dt \\
	& = \frac{1}{\pi}\dashint_{-1}^{1}\frac{1}{x-t}\frac{x-t-x}{\sqrt{1-t^2}}dt \\
	& = \int_{-1}^1 d\omega(t)  + x \frac{d}{dx} U^{\omega}(x) \\
	& = 1.
\end{align*} 
Thus $U^{\nu}(x) = x + c$, $x \in [-1,1]$, for some constant $c$.
By symmetry, it is clear that $U^\nu(0)=0$, and therefore $c=0$. 	
Hence our claim \eqref{eq:Unu} follows.

Then $\mu_{\alpha,0} - \frac{\varepsilon}{2} \nu$
is a measure with density \eqref{eq:psialphaeps}. It is a positive measure
on $[-1,1]$ if $|\varepsilon| \leq \varepsilon_{\alpha}$, and
then it is a probability measure since $\int d\nu = 0$ and $\int d\mu_{\alpha,0} = 1$.
It satisfies for some constant $\ell$,
\begin{align*}  
	2 U^{\mu_{\alpha,0} - \frac{\varepsilon}{2} \nu}(x)
	& = 2 U^{\mu_{\alpha,0}}(x) - \varepsilon U^{\nu}(x) \\
	& = - V\left(\frac{x}{\alpha}\right) + \ell - \varepsilon x, \qquad x \in [-1,1].
\end{align*}
Because of the variational conditions \eqref{eq:Umuconditions} this means that 
\[ \mu_{\alpha, \varepsilon} = \mu_{\alpha,0} - \frac{\varepsilon}{2} \nu\]
and the density \eqref{eq:psialphaeps} follows.
\end{proof}

Note that for $|\varepsilon| \leq \varepsilon_{\alpha}$,
\begin{equation} \label{eq:psialphaat0} 
	\psi_{\alpha, \varepsilon}(0) = \psi_{\alpha,0}(0) = 
	\frac{1}{2\alpha} + \frac{\sqrt{\alpha^2-1} - \arctan(\sqrt{\alpha^2-1})}{\alpha \pi}, 
\end{equation}
which can be seen from \eqref{eq:psialphaeps} and 
\eqref{eq:psialpha0}. Also note that $\psi_{\alpha, \varepsilon}(0) \to \frac{1}{2}$
as $\alpha \to 1+$.

Proposition \ref{prop2} (b) will follow from the following universality result
for the weights $(1-x^2)^{\pm 1/2} e^{-N V_{\alpha, \varepsilon}(x)}$.

\begin{proposition}\label{prop4}
Let $\alpha > 1$ and $|\varepsilon| < \varepsilon_{\alpha}$,
where $\varepsilon_{\alpha}$ is given by \eqref{eq:epsalpha}.
Let $\psi_{\alpha,\varepsilon}$ be the density of the 
equilibrium measure in $[-1,1]$ for the external field \eqref{eq:Valphaeps} on $[-1,1]$.

Then for each $x_0 \in (-1,1)$,
\begin{multline} \label{kernel:exponential-weights}
	\frac{1}{\psi_{\alpha, \varepsilon}(x_0) N} 
	K_{N}\left(x_0+\frac{x}{\psi_{\alpha, \varepsilon}(x_0) N}, 
		x_0+\frac{y}{\psi_{\alpha, \varepsilon}(x_0) N}; 
		(1-t^2)^{-1/2} e^{-N V_{\alpha, \varepsilon}(t)} \right) \\
	= \frac{\sin\pi (x-y)}{\pi (x-y)} + O\left(\frac{1}{N}\right)
\end{multline}
and
\begin{multline} \label{kernel:exponential-weights2}
	\frac{1}{\psi_{\alpha, \varepsilon}(x_0) N} 
	K_{N}\left(x_0+\frac{x}{\psi_{\alpha, \varepsilon}(x_0) N}, 
		x_0+\frac{y}{\psi_{\alpha, \varepsilon}(x_0) N}; 
		(1-t^2)^{1/2} e^{-N V_{\alpha, \varepsilon}(t)} \right) \\
	= \frac{\sin\pi (x-y)}{\pi (x-y)} + O\left(\frac{1}{N}\right)
\end{multline}
as $N \to \infty$.

The $O$-terms in \eqref{kernel:exponential-weights} and \eqref{kernel:exponential-weights2} 
are uniform for $x$ and $y$ in a compact subset of the real line, 
for $x_0$ in a compact subset of $(-1,1)$,
and for $\varepsilon$ in a compact subset of $(-\varepsilon_{\alpha}, \varepsilon_{\alpha})$.
\end{proposition}

The universality results \eqref{kernel:exponential-weights} and \eqref{kernel:exponential-weights2}
are well-known and they are in fact known under much more general conditions.
The interest for our present purposes lies in the fact that the $O$-terms are
uniform in $\varepsilon$.

Deift et al.~\cite{DKMVZ1} proved universality results for weights $e^{-nW}$ on the
real line and they developed the Riemann-Hilbert method for orthogonal polynomials \cite{Deift}
in order to do so. For nonvarying weights on $[-1,1]$ that are real analytic modifications
of a Jacobi weight, i.e., weights of the form
\begin{equation} \label{eq:Jacobitype} 
	h(x) (1-x)^{\alpha} (1+x)^{\beta} \end{equation}
where $h$ is real analytic and nowhere vanishing on $[-1,1]$, 
the universality was shown in \cite{KMVV}. 
In \cite{KuijlaarsNotes} it was emphasized that the Riemann-Hilbert method does not 
require any endpoint analysis for weights of the form  \eqref{eq:Jacobitype} 
with $\alpha, \beta \in \{-\frac{1}{2}, \frac{1}{2}\}$. 

Varying weights $h(x) (1-x)^{\alpha} (1+x)^{\beta} e^{-nW(x)}$ on the interval 
$[-1,1]$ have not been considered explicitly, as far as we are aware, but the 
Riemann-Hilbert analysis is very similar to the one in \cite{KMVV},
provided that $W$ is real analytic in a neighborhood of $[-1,1]$ with an
equilibrium measure having a density $\psi_W$ that is nowhere vanishing  on $(-1,1)$ 
and  has inverse square root behavior at both endpoints $\pm 1$.
Again, if $\alpha, \beta \in \{ - \frac{1}{2}, \frac{1}{2} \}$, then no
endpoint analysis in the Riemann-Hilbert method is needed.
This is a technical advantage, and it is the reason for the prefactors
$(1-x^2)^{\pm 1/2}$ in the weights in \eqref{kernel:exponential-weights} and
\eqref{kernel:exponential-weights2}.

In another development, Levin and Lubinsky  \cite[Theorem 1.1]{LevinLubinsky} 
established the sine kernel universality limits for very general exponential weights 
$h e^{-n W}$, with fixed $h$ and $W$ that do not depend on $n$.
We cannot apply this result directly, since we need uniformity in $\varepsilon$. 
In the same paper \cite[Theorem 1.2]{LevinLubinsky}, universality for exponential 
weights $h e^{-n W_n}$ with varying external fields $W_n$ is established 
under the assumption that  some control on the behavior of the Christoffel functions 
(i.e., the kernel on the diagonal) is known a priori. 
The behavior of Christoffel functions has been obtained by Totik 
\cite[Theorems 1.1 and 1.2]{Totik2000}, but again uniformity in parameters was not
included. 

Rather than trying to adapt the arguments of Totik to our situation, 
which deals with a very specific weight, we rely on the Riemann-Hilbert method
to prove  Proposition \ref{prop4}. The application of this method is standard by now,
but we will go through the analysis in order to verify the uniformity in $\varepsilon$,
see appendix \ref{sec:RHanalysis}.

\subsection{Proof of Proposition \ref{prop2} (b)}

We finally  show how part (b) of Proposition \ref{prop2} 
follows from Proposition  \ref{prop4}. 
\begin{proof}
Specializing \eqref{kernel:exponential-weights} to $x_0=0$, putting
\begin{equation} \label{eq:calphaplus} 
	c_{\alpha}^+ = \psi_{\alpha,\varepsilon}(0)=\psi_{\alpha,0}(0),
	\end{equation}
see \eqref{eq:psialphaat0}, and changing $x$ and $y$ to $ c_{\alpha}^+x$ and 
$c_{\alpha}^+ y$, we obtain
\[ 	\lim_{N\to\infty}	\frac{1}{N} 
	K_{N}\left(\frac{x}{N}, \frac{y}{N}; 
		(1-t^2)^{-1/2} e^{-N V_{\alpha, \varepsilon}(t)} \right) 
	= \frac{\sin\pi c_{\alpha}^+ (x-y)}{\pi (x-y)} 
\] 
uniformly for $\varepsilon$ in compact subsets of $(-\varepsilon_{\alpha}, \varepsilon_{\alpha})$.
By the uniformity, we obtain the same limit, if we let $\varepsilon = \varepsilon_R$ 
and let both $N, R \to \infty$ such that $\varepsilon_R \to 0$. Hence \eqref{eq:ass2a}
holds for the $+$-case. 

The weight $w_{R,\alpha}^-$ is supported on $[-\alpha^{-2}, \alpha^{-2}]$, see \eqref{eq:wRminus}.
After rescaling to $[-1,1]$ it is 
\begin{align*} 
	w_{R,\alpha}^-(\alpha^{-2} t) & = 
	 (1-t^2)^{1/2} e^{-N(V(\frac{t}{\alpha}) +\widetilde{\varepsilon}_R t)},
	\qquad t \in [-1,1],
\end{align*}
where 
\[
\widetilde{\varepsilon}_R=\alpha^{-2}\varepsilon_R. 
\]

The kernel transforms as
\begin{equation} \label{eq:KNwminus1} 
	K_N(x,y, w_{R,\alpha}^-) = 
	\alpha^2  K_N\left( \alpha^2 x, \alpha^2y; (1-t^2)^{1/2} 
	e^{-N V_{\alpha, \widetilde{\varepsilon}_R }(t)} \right).
	\end{equation}

From \eqref{kernel:exponential-weights2} with $x_0=0$, $x$ and $y$ replaced by
$\psi_{\alpha,\varepsilon}(0) x\alpha^2$ and $\psi_{\alpha,\varepsilon}(0) y\alpha^2$, we obtain
\[ 	\lim_{N\to\infty}	\frac{\alpha^2}{ N} 
	K_{N}\left(\frac{x\alpha^2}{ N}, \frac{y\alpha^2}{ N}; 
		(1-t^2)^{1/2} e^{-N V_{\alpha, \varepsilon}(t)} \right) 
	= \frac{\sin\pi c_{\alpha}^- (x-y)}{\pi (x-y)} 
\] 
uniformly for $\varepsilon$ in compact subsets of 
$(- \varepsilon_{\alpha}, \varepsilon_{\alpha})$,
where
\begin{equation} \label{eq:calphaminus} 
	c_{\alpha}^- = \alpha^2\psi_{\alpha,0}(0). 
	\end{equation}
Because of the uniformity in $\varepsilon$, we get 
\begin{equation} \label{eq:KNwminus2} 	
	\lim_{R\to\infty}	\frac{\alpha^2}{ N} 
	K_{N}\left(\frac{x\alpha^2}{N}, \frac{y\alpha^2}{ N}; 
		(1-t^2)^{1/2} e^{-N V_{\alpha, \widetilde{\varepsilon}_R }(t)} \right) 
	= \frac{\sin\pi c_{\alpha}^- (x-y)}{\pi (x-y)} 
\end{equation} 
if both $R, N \to \infty$ and $\widetilde{\varepsilon}_R \to 0$.
Combining \eqref{eq:KNwminus1} and \eqref{eq:KNwminus2} we obtain \eqref{eq:ass2a}
for the $-$-case as well.

The limits \eqref{eq:ass2b} are immediate from \eqref{eq:calphaplus}
and \eqref{eq:calphaminus}, since $\psi_{\alpha,0}(0) \to \frac{1}{2}$ as $\alpha \to 1+$.
\end{proof}

\section{Proof of Theorem \ref{mainthm:wR}} \label{sec:proofmain}

\begin{proof}
The reciprocal of the polynomial kernel 
$\widehat{K}_N(x,x;w)^{-1}$ (recall its definition in  \eqref{kernelwithoutweights})
is known in the theory of orthogonal polynomials as the Christoffel function.
It satisfies
\[ \frac{1}{\widehat{K}_N(x,x;w)} = 
	\min_{\substack{P(x)=1 \\ \deg P \leq N-1}}\int |P(t)|^2w(t)dt,
\]
where the minimum is taken over polynomials $P$ of degree at most 
$N-1$ that take the value $1$ at $x$.
From this extremal property and \eqref{eq:ass1a} the inequalities
\begin{equation} \label{eq:CFineq} 
	\widehat{K}_N(x,x; w_{R,\alpha}^+) \leq 
	\widehat{K}_N(x,x; w_{R}) \leq  \widehat{K}_N(x,x; w_{R,\alpha}^-) 
	\end{equation}
for $R > R_{\alpha}$, follow. 

We assume $\alpha$ is close enough to $1$ so that \eqref{eq:ass2a} holds. 
Because of 
 \eqref{kernelKn}-\eqref{kernelwithoutweights}, and \eqref{eq:ass1b}
we then also have the corresponding behavior for the $\widehat{K}_N$ kernels 
\begin{equation} \label{eq:Knhatuniv}
	\lim_{R \to\infty} \widehat{K}_N 
	\left( \frac{x}{N}, \frac{y}{N}; w_{R,\alpha}^{\pm} \right) =
	\frac{\sin \pi c_{\alpha}^{\pm} (x-y)}{\pi(x-y)}
	\end{equation}
and in particular if $y=x$,
\[ \lim_{R \to \infty} \frac{1}{N} 
	\widehat{K}_N\left( \frac{x}{N}, \frac{x}{N}; w_{R,\alpha}^{\pm} \right) 
	= c_{\alpha}^{\pm}. \]
By the inequalities \eqref{eq:CFineq} we get  
\[ c_\alpha^+ \leq \liminf_{R \to \infty} \frac{1}{N} 
	\widehat{K}_N\left( \frac{x}{N}, \frac{x}{N}; w_R \right) 
	\leq \limsup_{R \to \infty}  \frac{1}{N} 
	\widehat{K}_N\left( \frac{x}{N}, \frac{x}{N}; w_R \right) 
	\leq c_{\alpha}^-. \]
The parameter $\alpha$ can be taken arbitrarily close to $1$, and we find
because of the assumption \eqref{eq:ass2b} that
\begin{equation} \label{eq:KNdiag} 
	\lim_{R \to\infty} \frac{1}{N} 
	\widehat{K}_N\left( \frac{x}{N}, \frac{x}{N}; w_R \right) = \frac{1}{2} 
	\end{equation}
uniformly for $x$ in compacts.  
	
We next use an idea of D.S.~Lubinsky \cite[p. 919]{LubinskyAnnals} to estimate $\widehat{K}_N(x,y)$ 
for $x \neq y$ in terms of values on the diagonal. The inequality of Lubinsky is 
\begin{equation} \label{eq:DLineq}
	\frac{\left| \widehat{K}_N(x,y; w_{R,\alpha}^+) 
	- \widehat{K}_N(x,y; w_R) \right|}{\widehat{K}_N(x,x;w_R)}
	 		\leq 
	 	\sqrt{\frac{\widehat{K}_N(y,y; w_R)}{\widehat{K}_N(x,x;w_R)}}
	 	\sqrt{1 - \frac{\widehat{K}_N(x,x; w_{R,\alpha}^+)}
	 	{\widehat{K}_N(x,x; w_R)}}.
\end{equation}
In \eqref{eq:DLineq} we replace $x$ and $y$ by $x/N$ and $y/N$, respectively, 
and
let $R \to \infty$. Then the right-hand side of \eqref{eq:DLineq}
has the limit $\sqrt{1-2 c_{\alpha}^+}$ by \eqref{eq:ass2a} and 
\eqref{eq:KNdiag}. Applying \eqref{eq:KNdiag} also to the denominator
on the left-hand side of \eqref{eq:DLineq} we obtain
\begin{equation} \label{eq:KNhatcompare} 
	\limsup_{R \to \infty} \left| 
	\frac{1}{N} \widehat{K}_N \left(\frac{x}{N}, \frac{y}{N}; w_{R,\alpha}^+\right) 
	- \frac{1}{N} \widehat{K}_N\left(\frac{x}{N}, \frac{y}{N}; w_R\right) 
	\right| \leq \frac{\sqrt{1-2c_{\alpha}^+}}{2}. 
	\end{equation}

From \eqref{eq:ass1b} and \eqref{eq:ass2a} we have
\[ \lim_{R\to\infty} \frac{1}{N} \widehat{K}_N \left(\frac{x}{N}, \frac{y}{N};
	 w_{R,\alpha}^+\right) = \frac{\sin \pi c_{\alpha}^+(x-y)}{\pi(x-y)} \]
which when combined with \eqref{eq:KNhatcompare} leads to
\begin{equation} \label{eq:DLineqlimit} 
	\limsup_{R\to \infty} \left| 
	 \frac{\sin \pi c_{\alpha}^+(x-y)}{\pi(x-y)}
	- \frac{1}{N} \widehat{K}_N\left(\frac{x}{N}, \frac{y}{N}; w_R\right) 
	\right| \leq \frac{\sqrt{1-2c_{\alpha}^+}}{2}. 
	\end{equation}
Now let $\alpha \to 1$. Since $c_{\alpha}^+ \to 1/2$ as $\alpha \to 1$
we obtain from \eqref{eq:DLineqlimit}
\begin{equation} \label{eq:KNhatuniv} 
	\lim_{R \to \infty} \frac{1}{N} 
		\widehat{K}_N\left(\frac{x}{N}, \frac{y}{N}; w_R\right) 
		= 	 \frac{\sin  \frac{\pi}{2} (x-y)}{\pi(x-y)}.
	\end{equation}
 From \eqref{eq:ass1a} and \eqref{eq:ass1b} we see that
 $w_R(x/N) \to 1$ and $w_R(y/N) \to 1$ as $N \to \infty$, and
 we obtain 
\begin{equation} \label{eq:KNuniv0} 
	\lim_{N \to \infty} \frac{1}{N} 
		K_N\left(\frac{x}{N}, \frac{y}{N}; w_R\right) 
		= 	 \frac{\sin  \frac{\pi}{2} (x-y)}{\pi(x-y)},
	\end{equation}
which leads to \eqref{eq:KNuniv} if we replace $x$ by $2x$ and $y$ by $2y$.
\end{proof}

\appendix
\section{Riemann-Hilbert analysis and proof of Proposition \ref{prop4}} \label{sec:RHanalysis}.

In this section we prove Proposition \ref{prop4}  using the Riemann-Hilbert approach.
We apply the Riemann-Hilbert analysis to the weight
\[ W_N(x) = (1-x^2)^{-1/2} e^{-NV_{\alpha, \varepsilon}(x)}, \]
that appears in \eqref{kernel:exponential-weights}. The analysis for
the weight in \eqref{kernel:exponential-weights2} is similar.
The Riemann-Hilbert problem  associated with $W_N$ asks for a $2 \times 2$ matrix valued function
$Y : \mathbb C \setminus [-1,1] \to \mathbb C^{2 \times 2}$ such that
\begin{enumerate}
\item[Y1:] $Y$ is analytic on $z \in \mathbb{C} \setminus [-1,1]$ with continuous 
boundary values $Y_{\pm}$ on $(-1,1)$;
\item[Y2:] $Y_+ (x)= Y_-(x) \begin{pmatrix} 1 & W_N(x)\\ 0 & 1  \end{pmatrix}$ 
for every $x \in (-1,1)$;
\item[Y3:] $Y(z)=\left(I+O\left( \frac{1}{z} \right) \right)
\begin{pmatrix} z^N & 0 \\ 0 & z^{-N} \end{pmatrix}$ as $z \to \infty$;
\item[Y4:] $Y(z)= O\begin{pmatrix} 1 & |z\pm 1|^{-1/2} \\ 
1 & |z\pm 1|^{-1/2} \end{pmatrix}$ as $z \to \pm 1$.
\end{enumerate}

The Riemann-Hilbert problem has a unique solution, where $Y_{11}$ is the monic
orthogonal polynomial of degree $N$ for the weight $W_N$. The kernel $K_N$ is
given in terms of $Y$ as 
\begin{equation} \label{eq:kernelKNinY} 
	K_{N}(x,y; W_N)= 
	\frac{\sqrt{W_N(x) W_N(y)}}{2\pi i(x-y)} 
	\begin{pmatrix} 0&  1 \end{pmatrix}Y_+(y)^{-1}Y_+(x)
	\begin{pmatrix} 1\\ 0 \end{pmatrix},
\end{equation} 
for $x, y \in (-1,1)$.

Let $\psi_{\alpha, \varepsilon}$ be the density of the equilibrium measure
as in \eqref{eq:psialphaeps}, and let $\ell_{\alpha}$ be the constant in the variational
condition \eqref{eq:Umuconditions} for the external field $V_{\alpha, \varepsilon}$.
The constant $\ell_{\alpha}$ does not depend on $\varepsilon$. It can be explicitly
calculated, but the precise value is of no interest to us here.

Following \cite{DKMVZ1,DeiftVenakidesZhou}, we define the $g$-function
\[
	g(z):=\int_{-1}^1 \log(z-t) \psi_{\alpha, \varepsilon}(t) dt,
	\qquad z\in \mathbb{C} \setminus [-\infty,1],
\]
with the principal branch of the logarithm.  Then $T$ defined by  
\begin{equation} \label{eq:defT}
	T(z)=\begin{pmatrix} 	e^{-\ell_{\alpha}} & 0 \\	0&  e^{\ell_{\alpha}} \end{pmatrix}
	Y(z)
	\begin{pmatrix} e^{-N(g(z)-\ell_{\alpha})} & 0 \\ 0&  e^{N(g(z)-\ell_{\alpha})} \end{pmatrix}
\end{equation}
satisfies 
\begin{enumerate}
\item[T1:] $T$ is analytic on $\mathbb{C} \setminus [-1,1]$ with continuous boundary 
	values $T_{\pm}$ on $(-1,1)$;
\item[T2:] $T_+ (x)= T_-(x) J_T(x)$ for $x \in (-1,1)$, where
\begin{equation} \label{eq:JT}
	J_T(x) = 
	\begin{pmatrix} e^{-2\pi i N \int_x^1 \psi_{\alpha,\varepsilon}(s) ds} & \frac{1}{\sqrt{1-x^2}}  \\
	0 & e^{2\pi i N \int_x^1 \psi_{\alpha, \varepsilon}(s) ds}  \end{pmatrix}; 
	\end{equation}
\item[T3:] $T(z)=I+O\left( \frac{1}{z} \right)$ as $z \to \infty$;
\item[T4:] $T$ has the same behavior as $Y$ as $z \to \pm 1$. 
\end{enumerate}

Define
\begin{align} \nonumber
	\rho_{\alpha, \varepsilon}(x) & := 2 \pi \sqrt{1-x^2}\psi_{\alpha, \varepsilon}(x) \\ \label{eq:rhoalphaeps}
		&  = \varepsilon_{\alpha} -  \varepsilon x 	+ 
		\frac{2}{\alpha} \sqrt{1-x^2} \arctan \left( \frac{\sqrt{1-x^2}}{\sqrt{\alpha^2-1}}\right),
			\qquad x \in [-1,1]
\end{align}
where we recall that $\varepsilon_{\alpha}$ is given by \eqref{eq:epsalpha}.
Then $\rho_{\alpha, \varepsilon}$ has an analytic continuation to 
\[ U_{\alpha} := \mathbb C \setminus ((-\infty,-\alpha] \cup [\alpha, \infty)) \] 
which contains
the interval $[-1,1]$ in its interior. We also use $\rho_{\alpha, \varepsilon}$
for the analytic continuation. 
Let
\[ \varphi(z) = z + (z^2-1)^{1/2}, \qquad z \in \mathbb C \setminus [-1,1] \]
be the conformal map from the exterior of $[-1,1]$ to the exterior of the unit disk.
For each $\tau \in (1, \varphi(\alpha))$, the curve 
\[ \Gamma_{\tau} = \{ z \in \mathbb C \setminus [-1,1] \mid \varphi(z) = \tau \} \]
is an ellipse around $[-1,1]$ lying in $U_{\alpha}$. 
It is then clear that there is an $M$ such that
\begin{equation}  \label{eq:rhobound1}
	| \rho_{\alpha, \varepsilon}'(z)| \leq M, \qquad z \in \inte(\Gamma_{\tau})
	\end{equation} 
and $M$ is independent of $\varepsilon \in [-\varepsilon_{\alpha}, \varepsilon_{\alpha}]$, 
which is easily seen from the simple way
that $\rho_{\alpha, \varepsilon}$ depends on $\varepsilon$, see \eqref{eq:rhoalphaeps}.
Also from \eqref{eq:rhoalphaeps}
\begin{equation} \label{eq:rhobound2}
	\rho_{\alpha, \varepsilon}(x) \geq \varepsilon_{\alpha} - |\varepsilon|,
		\qquad x \in [-1,1].
\end{equation}

We let
\begin{align} \label{eq:defxi}
	\xi(z)  = \int_1^z \frac{\rho_{\alpha, \varepsilon}(s)}{ (s^2-1)^{1/2}}ds, 
	\qquad z\in U_{\alpha} \setminus (-\infty,1], 
\end{align} 
where $(s^2-1)^{1/2}$ is analytic for $ s \in \mathbb C \setminus [-1,1]$
and positive for $s > 1$. The contour of integration in \eqref{eq:defxi} is  the
line segment from $1$ to $z$ in the complex plane. It can be checked that 
$e^{\xi(z)}$ has an analytic extension to all of $U_{\alpha} \setminus [-1,1]$. 
Then  for $x \in (-1,1)$,
\[ \xi_+(x) = - \xi_-(x) = 2\pi i \int_x^1 \psi_{\alpha, \varepsilon}(s) ds ,\]
and so we can write the jump matrix \eqref{eq:JT} for $T$ in terms of $\xi_{\pm}$. 
There is a factorization
\begin{align} \label{eq:JTfactor}
\begin{split} 
	 J_T(x) ={}& \begin{pmatrix} 1 & 0 \\ i(x^2-1)^{1/2}_- e^{-N\xi_-(x)} & 1 \end{pmatrix}
	\begin{pmatrix} 0 & \frac{1}{\sqrt{1-x^2}} \\ - \sqrt{1-x^2} & 0 \end{pmatrix}\\
	&\times
	\begin{pmatrix} 1 & 0 \\ -i(x^2-1)^{1/2}_+ e^{-N\xi_+(x)} & 1 \end{pmatrix}.
\end{split}
\end{align}

We take $\tau > 1$ close to $1$ and define $S$ by 
\begin{equation} \label{eq:defS}
	S(z)=\begin{cases}
	T(z), & z\in \ext(\Gamma_\tau),\\
	T(z) \begin{pmatrix} 1 & 0 \\ i (z^2-1)^{1/2} e^{-N\xi(z)} & 1 \end{pmatrix}, 
	& z\in \inte(\Gamma_\tau)\setminus [-1,1].
\end{cases}
\end{equation}
Then $S$ satisfies the following Riemann-Hilbert problem.
\begin{enumerate}
\item[S1:] $S$ is analytic on $\mathbb{C} \setminus (\Gamma_\tau \cup [-1,1])$ 
with continuous boundary values on $\Gamma_{\tau} \cup (-1,1)$;
\item[S2:] $S_+ (x)= S_-(x) 
	\begin{pmatrix} 0 & \frac{1}{\sqrt{1-x^2}} \\ -\sqrt{1-x^2}  & 0 
	\end{pmatrix}$ for $x \in (-1,1)$, and 
	
	$S_+ = S_- J_S$ on $\Gamma_{\tau}$ where
	\begin{equation} \label{eq:JS}
	\begin{pmatrix} 1 & 0 \\ -i (z^2-1)^{1/2} e^{-N\xi(z)} & 1 \end{pmatrix};
	\end{equation}
\item[S3:] $S(z) = I + O\left(\frac{1}{z}\right)$ as $z \to \infty$; 
\item[S4:] $S$ has the same behavior as $T$ as $z \rightarrow \pm 1$. 
\end{enumerate}

Let $z \in \Gamma_{\tau}$. There is $\theta \in [0,2\pi]$ such that
$\varphi(z) = \tau e^{i\theta}$. Then by \eqref{eq:defxi}
\begin{align} \nonumber
	\xi(z) & =   \int_1^z
	 \frac{\rho_{\alpha, \varepsilon}(s) - \rho(\cos \theta)}{
	 (s^2-1)^{1/2}} ds  + \rho(\cos \theta) \int_1^z \frac{1}{(s^2-1)^{1/2}} ds \\
	& =  \int_1^z
	 \frac{\rho_{\alpha, \varepsilon}(s) - \rho(\cos \theta)}{
	 (s^2-1)^{1/2}} ds  + \rho(\cos \theta) \log \varphi(z). 
	 \label{eq:Rexi1} 
	 \end{align} 
Note that the integrand in the integral in the right-hand side of \eqref{eq:Rexi1}
is purely imaginary for $s \in [-1,1]$. For the real part of $\xi(z)$ we may then
start integrating at any other point in $[-1,1]$ instead of $1$. We choose $\cos\theta$
and since $|\varphi(z)| = \tau$, we obtain from \eqref{eq:Rexi1}
\begin{align} \label{eq:Rexi2}
	\Re \xi(z) = \Re \left( \int_{\cos \theta}^z
	 \frac{\rho_{\alpha, \varepsilon}(s) - \rho_{\alpha, \varepsilon}(\cos \theta)}{
	 (s^2-1)^{1/2}} ds \right) + \rho_{\alpha, \varepsilon}(\cos \theta ) \log  \tau .
	 \end{align}
We make the change of variables $s = \varphi^{-1}(w) = \frac{1}{2}(w + \frac{1}{w})$,
and then \eqref{eq:Rexi2} leads to
\begin{equation} \label{eq:Rexi3}
	\Re \xi(z) =	\Re \left( \int_{e^{i\theta}}^{\tau e^{i \theta}} 
	 \frac{\rho_{\alpha, \varepsilon}( \frac{1}{2} (w+\frac{1}{w})) 
	 - \rho_{\alpha, \varepsilon}(\cos \theta)}{w} dw	 \right) + 
	 \rho_{\alpha, \varepsilon}(\cos \theta) \log \tau.
	 \end{equation}
The first term in \eqref{eq:Rexi3} is $\geq -M(\tau-1)^2$ because
of \eqref{eq:rhobound1}. The second term is estimated by \eqref{eq:rhobound2} and
we find for $z \in \Gamma_{\tau}$,
\begin{equation} \label{eq:Rexi4}
	\Re \xi(z) \geq (\varepsilon_{\alpha} - |\varepsilon|) \log \tau
		 -M\left(\tau-1 \right)^2.
\end{equation}
Thus if $\varepsilon$ is in a compact subset of $(-\varepsilon_{\alpha}, \varepsilon_{\alpha})$
then there are $c > 0$ and  $\tau > 1$ such that for all $z$ in a neighborhood of $\Gamma_{\tau}$,
\begin{equation} \label{eq:Rexi5}
	\Re \xi(z) \geq c > 0.
	\end{equation}
For such a $\tau > 1$, the jump matrix on $\Gamma_{\tau}$ in the Riemann-Hilbert problem
for $S$ tends to the identity matrix at an exponential rate as $N \to \infty$.

Define now
\begin{equation} \label{eq:defM}
	M(z)=  \begin{pmatrix} 1 & -\frac{1}{i (z^2-1)^{1/2}} \\
	\frac{z- (z^2-1)^{1/2}}{2i} &  \frac{z+ (z^2-1)^{1/2}}{2 (z^2-1)^{1/2}}
\end{pmatrix}, \qquad z \in \mathbb C \setminus [-1,1].
\end{equation}
Then it is easy to verify that  
\begin{enumerate}
\item[M1:] $M$ is analytic on $\mathbb{C} \setminus [-1,1]$ with continuous boundary 
values on $(-1,1)$;
\item[M2:] $M$ has the same jump as $S$ on $(-1,1)$;
\item[M3:] $M(z) = I + O(\frac{1}{z})$ as $z \rightarrow \infty$; 
\item[M4:] $M(z)$ has the same behavior as $S(z)$ as $z \rightarrow \pm 1$. 
\end{enumerate} 

We now introduce the final transformation 
\begin{equation} \label{eq:defR}
	R(z)=S(z)M^{-1}(z), \qquad z \in \mathbb C \setminus (\Gamma_{\tau} \cup [-1,1]).
\end{equation}
Then $R$ has analytic continuation across $[-1,1]$, and 
$R$ satisfies the following Riemann-Hilbert problem.
\begin{enumerate}
\item[R1:] $R$ is analytic on $\mathbb{C} \setminus \Gamma_{\tau}$;
\item[R2:] $R_+ = R_- J_R$ on $\Gamma_{\tau}$, where $J_R = M J_S M^{-1}$, see \eqref{eq:JS};
\item[R3:] $R(z)=I+O(\frac{1}{z})$ as $z \to \infty$.
\end{enumerate}

The jump matrix $J_R$ satisfies because of \eqref{eq:JS} and \eqref{eq:Rexi5},
\begin{align} \label{eq:JR}
	J_R(z) & = M(z) \begin{pmatrix} 1 & 0 \\
	-i(z^2-1)^{1/2} e^{-N\xi(z)} & 1 \end{pmatrix} M^{-1}(z) \\
	& = I+O\left(e^{-cN} \right), \qquad \text{ as } N \to \infty, \nonumber
\end{align}
uniformly for $\varepsilon$ in a compact subset of $(-\varepsilon_{\alpha}, \varepsilon_{\alpha})$.

Then by standard estimates (see e.g.\ 
\cite[Sec. 5.9]{AptekarevSbornik} and \cite[Thm. 3.1]{KuijlaarsNotes})  we have
that 
\begin{equation} \label{eq:Rzasymp} 
R(z)=I+O(e^{-cN}) \qquad \text{ as } N \to \infty,	
\end{equation}
uniformly for $z \in \mathbb C \setminus \Gamma_{\tau}$ and uniformly
for $\varepsilon$ in compact subset of $(-\varepsilon_{\alpha}, \varepsilon_{\alpha})$.

To obtain \eqref{kernel:exponential-weights} we then follow the effect
of the transformations $Y \mapsto T \mapsto S \mapsto R$ on the 
expression \eqref{eq:kernelKNinY} for the kernel $K_N$. The computations
are similar to what is done in  \cite{Deift,DKMVZ1}, and in particular 
in  \cite[Section 2.7]{BleherLiechty}. Since \eqref{eq:Rzasymp} is
uniform in $\varepsilon$ the $O$-term in \eqref{kernel:exponential-weights} is
uniform for $\varepsilon$ in compact subsets of $(-\varepsilon_{\alpha}, \varepsilon_{\alpha})$
as well.

\begin{remark}
The proof for \eqref{kernel:exponential-weights2} follows along the same lines.
Instead of $\frac{1}{\sqrt{1-x^2}}$ one has $\sqrt{1-x^2}$ in the $12$-entry 
of the jump matrix $J_T$ in \eqref{eq:JT}.
The transformations are similar, but  slightly different, and they 
ultimately lead to an expression for the jump matrix $J_R$ as
\[ J_R(z) = M(z) \begin{pmatrix}  1 & 0 \\ i(z^2-1)^{-1/2} e^{-N \xi(z)} & 1 \end{pmatrix}
	M^{-1}(z), \]
compare with \eqref{eq:JR}, with a different $M$. Again $J_R(z) = I + O(e^{-cN})$
for $z \in \Gamma_{\tau}$, and the rest of the argument is the same.
\end{remark}

\begin{remark}
If the square root factors $(1-x^2)^{\pm 1/2}$ are missing  from the weight,
then the Riemann-Hilbert analysis becomes more complicated. Then the $12$-entry of the jump
matrix $J_T$ is $1$, and there is a factorization 
\begin{equation}
	J_T = \begin{pmatrix} 1 & 0 \\ e^{-N \xi_-} & 1 \end{pmatrix}
	\begin{pmatrix} 0 & 1 \\ -1 & 0 \end{pmatrix} 
	\begin{pmatrix} 1 & 0 \\ e^{-N \xi_+} & 1 \end{pmatrix},
	\end{equation}
compare with \eqref{eq:JTfactor}, and the jump for $T$ is  written as
\[ \left(T 
	\begin{pmatrix} 1 & 0 \\ -e^{-N \xi} & 1 \end{pmatrix} \right)_+
		= \left(T 
	\begin{pmatrix} 1 & 0 \\ e^{-N \xi} & 1 \end{pmatrix} \right)_-
	\begin{pmatrix} 0 & 1 \\ -1 & 0 \end{pmatrix}.
\]
This suggests to define 
$S = T 
	\begin{pmatrix} 1 & 0 \\ -e^{-N \xi} & 1 \end{pmatrix}$
	in the region inside $\Gamma_{\tau}$ in the upper half plane, 
and
$S = 	T 
	\begin{pmatrix} 1 & 0 \\ e^{-N \xi} & 1 \end{pmatrix}$
	in the region inside $\Gamma_{\tau}$ in the lower half plane.
However,  $S$ has a jump not only on $\Gamma_{\tau} \cup [-1,1]$, but
also on the two intervals  $[-\alpha, -1]$ and $[1,\alpha]$,
and a further analysis is necessary to handle these jumps.

\end{remark}

\section{Proof of  \eqref{eq:pnsineprocess}} \label{appendixB}
In this section we give a proof of the estimate \eqref{eq:pnsineprocess}, for which we have been unable to find an exact reference. 

We actually prove a somewhat stronger statement. For every $\varepsilon > 0$
we have almost surely
\begin{equation} \label{eq:toprove} 
	p_n = n + O(|n|^{1/2} \log^{1+\varepsilon}|n|)  \quad
	\text{ as } n \to \pm \infty. \end{equation}

Let $N(L)$ be the random variable giving the number of points
from the sine point process inside the interval $[0,L]$.
Then $\mathbb E[N(L)] = L$ and  the variance of $N(L)$ satisfies 
\begin{equation} \label{eq:Varestimate}
	\Var(N(L)) = \frac{1}{\pi^2} \log L + O(1) \qquad \text{as } L \to \infty. \end{equation}
see \cite[formula (8)]{CostinLebowitz}.
Because of the ordering (1.5) of the points $p_n$, we have
$p_n > L$ if and only if $N(L) \leq n$, and therefore if $L \geq n \geq 1$,
\begin{align} \nonumber 
	\mathbb P( p_n > L) & = \mathbb P(N(L) \leq n) \\ \nonumber
		& \leq \mathbb P( |N(L) - L| \geq L-n) \\
	& \leq \frac{\Var(N(L))}{(L-n)^2}, \label{eq:Pestimate1}
	\end{align}
by Chebyshev's inequality.
Taking $L = n + n^{1/2} \log^{1+\varepsilon} n$ for some 
fixed $\varepsilon > 0$, we easily obtain from 
\eqref{eq:Varestimate} and \eqref{eq:Pestimate1} that
\begin{equation*} 
	\mathbb P(p_n > n + n^{1/2} \log^{1+ \varepsilon} n)
	\leq \frac{1}{n \log^{1 + 2 \varepsilon} n} 
\end{equation*}
for $n$ large enough.
The series $\sum\limits_{n=2}^{\infty} \frac{1}{n \log^{1+2 \varepsilon} n}$ converges,
and so by the Borel-Cantelli lemma, we almost surely have
\[ p_n \leq n + n^{1/2} \log^{1 + \varepsilon} n 
 \qquad \text{for $n$ large enough}. \]
With a similar argument we obtain the almost sure lower bound 
$p_n \geq n - n^{1/2} \log^{1 +  \varepsilon} n$ for $n$ large enough,
and this proves \eqref{eq:toprove} for $n \to \infty$.

The a.s.\ limit for $n \to -\infty$ follows by symmetry.

\section*{Acknowledgements}
We are very grateful to Alexander Bufetov for posing the very interesting Problem
\ref{prob:Bufetov} that led to this work.   

The first author is supported by long term structural funding-Methusalem 
grant of the Flemish Government, by the Belgian Interuniversity Attraction 
Pole P07/18,  by KU Leuven Research Grant OT/12/073, and by FWO Flanders 
projects G.0864.16 and EOS 30889451.

The second author is very grateful to Arno Kuijlaars and Walter Van Assche for helping support his research stay at KU Leuven during the Fall of 2016. This stay was also partially supported by a grant from The University of Mississippi Office of Research and Sponsored Programs.

\end{document}